\documentclass[10pt, reqno]{amsart}
\usepackage{amsmath, amsthm, amscd, amsfonts, amssymb, graphicx, color}
\usepackage[bookmarksnumbered, colorlinks, plainpages]{hyperref}
\hypersetup{colorlinks=true,linkcolor=red, anchorcolor=green, citecolor=cyan, urlcolor=red, filecolor=magenta, pdftoolbar=true}

\newtheorem{theorem}{Theorem}[section]
\newtheorem{lemma}[theorem]{Lemma}
\newtheorem{proposition}[theorem]{Proposition}
\newtheorem{corollary}[theorem]{Corollary}
\theoremstyle{definition}

\newtheorem{example}[theorem]{Example}

\theoremstyle{remark}
\newtheorem{remark}[theorem]{Remark}
\numberwithin{equation}{section}

\begin{document}

\setcounter{page}{1}

\title[Characterizations of BMO and Lipschitz spaces]{Some characterizations of BMO and Lipschitz spaces in the Schr\"{o}dinger setting}

\author[H. Wang]{Cong Chen and Hua Wang}

\address{School of Mathematics and System Science, Xinjiang University, Urumqi 830046, P. R. China}
\email{\textcolor[rgb]{0.00,0.00,0.84}{3221043817@qq.com}}
\email{\textcolor[rgb]{0.00,0.00,0.84}{wanghua@pku.edu.cn}}
\dedicatory{Dedicated to the memory of Li Xue}

\let\thefootnote\relax\footnote{Copyright 2018 by the Tusi Mathematical Research Group.}



\date{}

\begin{abstract}
We consider the Schr\"{o}dinger operator $\mathcal{L}=-\Delta+V$ on $\mathbb R^d$, $d\geq3$, where the nonnegative
potential $V$ belongs to the reverse H\"{o}lder class $RH_s$ for some $s\geq d/2$. A real-valued function $f\in L^1_{\mathrm{loc}}(\mathbb R^d)$ belongs to the (BMO) space $\mathrm{BMO}_{\rho,\theta}(\mathbb R^d)$ with $0<\theta<\infty$ if
\begin{equation*}
\|f\|_{\mathrm{BMO}_{\rho,\theta}}
:=\sup_{B(x_0,r)}\bigg(1+\frac{r}{\rho(x_0)}\bigg)^{-\theta}\bigg(\frac{1}{|B(x_0,r)|}\int_{B(x_0,r)}\big|f(x)-f_{B}\big|\,dx\bigg),
\end{equation*}
where the supremum is taken over all balls $B(x_0,r)\subset\mathbb R^d$, $\rho(\cdot)$ is the critical radius function in the Schr\"{o}dinger context and
\begin{equation*}
f_{B}:=\frac{1}{|B(x_0,r)|}\int_{B(x_0,r)}f(y)\,dy.
\end{equation*}
A real-valued function $f\in L^1_{\mathrm{loc}}(\mathbb R^d)$ belongs to the (Lipschitz) space $\mathrm{Lip}_{\beta}^{\rho,\theta}(\mathbb R^d)$ with $0<\beta<1$ and $0<\theta<\infty$ if
\begin{equation*}
\|f\|_{\mathrm{Lip}_{\beta}^{\rho,\theta}}
:=\sup_{B(x_0,r)}\bigg(1+\frac{r}{\rho(x_0)}\bigg)^{-\theta}
\bigg(\frac{1}{|B(x_0,r)|^{1+\beta/d}}\int_{B(x_0,r)}\big|f(x)-f_{B}\big|\,dx\bigg).
\end{equation*}
It can be easily seen that $\mathrm{BMO}_{\rho,\theta}(\mathbb R^d)$ (or $\mathrm{Lip}_{\beta}^{\rho,\theta}(\mathbb R^d)$) is a function space which is larger than the classical BMO (or Lipschitz) space. In this paper, we give some new characterizations of BMO and Lipschitz spaces associated with the Schr\"{o}dinger operator $\mathcal{L}$. We extend some previous works of Bongioanni--Harboure--Salinas and Liu--Sheng to the weighted case. The classes of weights considered here are larger than the classical Muckenhoupt classes.
\end{abstract}
\maketitle

\section{\textbf{Introduction and preliminaries}}
In this paper, we are concerned with BMO and Lipschitz spaces in the Schr\"{o}dinger setting. We shall establish some new characterizations of these function spaces. Let $d\geq3$ be a positive integer and $\mathbb R^d$ be the $d$-dimensional Euclidean space, and let $V:\mathbb R^d\rightarrow\mathbb R$, $d\geq3$, be a non-negative locally integrable function which belongs to the \emph{reverse H\"older class} $RH_s(\mathbb R^d)$ for some exponent $s\in(1,\infty]$; i.e., there exists a positive constant $C=C(s,V)>0$ such that the following \emph{reverse H\"older inequality}
\begin{equation*}
\bigg(\frac{1}{|B|}\int_B V(y)^s\,dy\bigg)^{1/s}\leq C\cdot\bigg(\frac{1}{|B|}\int_B V(y)\,dy\bigg)
\end{equation*}
holds for every ball $B$ in $\mathbb R^d$, with the usual modification made when $s=\infty$. In particular, if $V$ is a nonnegative polynomial, then $V\in RH_{\infty}(\mathbb R^d)$. Let us consider the \emph{Schr\"{o}dinger operator}
\begin{equation*}
\mathcal{L}:=-\Delta+V \quad \mbox{on}~~~ \mathbb R^d,
\end{equation*}
where $\Delta$ is the standard Laplacian. As in \cite{shen}, for any given $V\in RH_s(\mathbb R^d)$ with $s\geq d/2$ and $d\geq3$, we introduce the \emph{critical radius function} $\rho(x)=\rho(x;V)$ which is given by
\begin{equation}\label{rho}
\rho(x):=\sup\bigg\{r>0:\frac{1}{r^{d-2}}\int_{B(x,r)}V(y)\,dy\leq1\bigg\},\quad x\in\mathbb R^d,
\end{equation}
where $B(x,r)$ denotes the open ball with the center at $x$ and radius $r$. It is well known that this auxiliary function satisfies $0<\rho(x)<\infty$ for any $x\in\mathbb R^d$ under the above assumption on $V$ (assume that $V\not\equiv0$, see \cite{shen}).
\begin{example}
The Schr\"{o}dinger operator $\mathcal{L}=-\Delta+V$ can be viewed as a perturbation of the Laplacian.
\begin{enumerate}
  \item When $V=1$, we obtain $\rho(x)=1$ for any $x\in\mathbb R^d$.
  \item When $V(x)=|x|^2$ and $\mathcal{L}$ becomes the Hermite operator, we obtain $\rho(x)\approx(1+|x|)^{-1}$.
\end{enumerate}
\end{example}
We need the following known result concerning the critical radius function \eqref{rho}, which was proved by Shen in \cite{shen}.
\begin{lemma}[\cite{shen}]\label{N0}
If $V\in RH_s(\mathbb R^d)$ with $s\geq d/2$ and $d\geq3$, then there exist two positive constants $C_0\geq 1$ and $N_0>0$ such that
\begin{equation}\label{com}
\frac{\,1\,}{C_0}\bigg(1+\frac{|x-y|}{\rho(x)}\bigg)^{-N_0}\leq\frac{\rho(y)}{\rho(x)}
\leq C_0\bigg(1+\frac{|x-y|}{\rho(x)}\bigg)^{\frac{N_0}{N_0+1}}
\end{equation}
for all $x,y\in\mathbb R^d.$
\end{lemma}
By a weight, we always mean a non-negative locally integrable function $\omega$ on $\mathbb R^d$. Given a Lebesgue measurable set $E\subset\mathbb R^d$ and a weight $\omega$, we use the notation $|E|$ to denote the Lebesgue measure of $E$ and $\omega(E)$ to denote the weighted measure of $E$,
\begin{equation*}
\omega(E):=\int_E \omega(x)\,dx.
\end{equation*}
For any given ball $B=B(x_0,r)$ and $\lambda\in(0,\infty)$, we will write $\lambda B$ for the $\lambda$-dilate ball, which is the ball with the same center $x_0$ and radius $\lambda r$; that is $\lambda B=B(x_0,\lambda r)$. Similarly, $Q(x_0,r)$ denotes the cube centered at $x_0$ and with the sidelength $r$. Here and in what follows, only cubes with sides parallel to the coordinate axes are considered, and $\lambda Q=Q(x_0,\lambda r)$. We shall consider two classes of weights that are given in terms of the critical radius function \eqref{rho}. As in \cite{bong1} (see also \cite{bong2}), we say that a weight $\omega$ belongs to the class $A^{\rho,\theta}_p(\mathbb R^d)$ for $1<p<\infty$ and $0<\theta<\infty$, if there is a positive constant $C>0$ such that for all balls $B=B(x_0,r)\subset\mathbb R^d$ with $x_0\in\mathbb R^d$ and $r\in(0,\infty)$,
\begin{equation*}
\bigg(\frac{1}{|B|}\int_B \omega(x)\,dx\bigg)^{1/p}\bigg(\frac{1}{|B|}\int_B \omega(x)^{-{p'}/p}\,dx\bigg)^{1/{p'}}
\leq C\cdot\bigg(1+\frac{r}{\rho(x_0)}\bigg)^{\theta},
\end{equation*}
where $p'$ denotes the \emph{conjugate index} of $p$, namely, $1/p+1/{p'}=1$. For $p=1$ and $0<\theta<\infty$, we also say that a weight $\omega$ belongs to the class $A^{\rho,\theta}_1(\mathbb R^d)$, if there is a positive constant $C>0$ such that for all balls $B=B(x_0,r)\subset\mathbb R^d$ with $x_0\in\mathbb R^d$ and $r\in(0,\infty)$,
\begin{equation*}
\frac1{|B|}\int_B \omega(x)\,dx\leq C\cdot\bigg(1+\frac{r}{\rho(x_0)}\bigg)^{\theta}\underset{x\in B}{\mbox{ess\,inf}}\;\omega(x).
\end{equation*}
Since
\begin{equation}\label{cc}
1\leq\bigg(1+\frac{r}{\rho(x_0)}\bigg)^{\theta_1}\leq\bigg(1+\frac{r}{\rho(x_0)}\bigg)^{\theta_2}
\end{equation}
whenever $0<\theta_1<\theta_2<\infty$, then for given $p$ with $1\leq p<\infty$, by definition, we have
\begin{equation*}
A_p(\mathbb R^d)\subset A^{\rho,\theta_1}_p(\mathbb R^d)\subset A^{\rho,\theta_2}_p(\mathbb R^d),
\end{equation*}
where $A_p(\mathbb R^d)$ denotes the classical Muckenhoupt class (see \cite{muck} and \cite{coi}). For any given $1\leq p<\infty$, as the classes $A^{\rho,\theta}_p(\mathbb R^d)$ increase with respect to $\theta$, it is natural to define
\begin{equation*}
A^{\rho,\infty}_p(\mathbb R^d):=\bigcup_{\theta>0}A^{\rho,\theta}_p(\mathbb R^d).
\end{equation*}
Consequently, one has the inclusion relation
\begin{equation*}
A_p(\mathbb R^d)\subset A^{\rho,\infty}_p(\mathbb R^d).
\end{equation*}
However, the converse is not true, it is easy to check that the above inclusion is strict. In fact, if $\omega\in A_p(\mathbb R^d)$ for some $p>1$, then $\omega(x)\,dx$ is a doubling measure (see \cite{grafakos,grafakos2}), i.e., there exists a universal constant $C>0$ such that for any ball $B$
\begin{equation*}
\omega(2B)\leq C\omega(B).
\end{equation*}
If $\omega\in A^{\rho,\theta}_p(\mathbb R^d)$ for some $p>1$ and $\theta>0$, then $\omega(x)\,dx$ may not be a doubling measure. For example, the weight
\begin{equation*}
\omega_{\gamma}(x)=(1+|x|)^{\gamma}\in A^{\rho,\theta}_p(\mathbb R^d)\quad \mbox{for any}~~ \gamma>d(p-1)
\end{equation*}
provided that $V=1$ and $\rho(\cdot)\equiv1$. It is easy to see that such choice of $\omega_{\gamma}$ yields $\omega_{\gamma}(x)\,dx$ is not a doubling measure, hence it does not belong to $A_q(\mathbb R^d)$ for any $1<q<\infty$. The situation is more complicated. We can define (generalized) doubling classes of weights adapted to the Schr\"{o}dinger context, see \cite{bong9} and \cite{bong10}, for example.

For any given $\theta>0$, let us introduce the maximal operator which is given in terms of the critical radius function \eqref{rho}.
\begin{equation*}
\mathcal{M}_{\rho,\theta}f(x):=\sup_{r>0}\bigg(1+\frac{r}{\rho(x)}\bigg)^{-\theta}\frac{1}{|B(x,r)|}\int_{B(x,r)}|f(y)|\,dy,\quad x\in\mathbb R^d.
\end{equation*}
The classes $A^{\rho,\infty}_p(\mathbb R^d)$ are closely connected with the family of maximal operators $\mathcal{M}_{\rho,\theta}$. Observe that a weight $\omega$ belongs to the class $A^{\rho,\infty}_1(\mathbb R^d)$ if and only if there exists a positive number $\theta>0$ such that $\mathcal{M}_{\rho,\theta}(\omega)(x)\leq C\cdot\omega(x)$, for a.e.~$x\in\mathbb R^d$, where the constant $C>0$ is independent of $\omega$. Moreover, as in the classical setting, the classes of weights are characterized by the weighted boundedness of the corresponding maximal operators. Let $1<p<\infty$. It can be shown that $\omega\in A^{\rho,\infty}_p(\mathbb R^d)$ if and only if there exists a positive number $\theta>0$ such that $\mathcal{M}_{\rho,\theta}$ is bounded on $L^p(\omega)$ (see \cite{bong6} and \cite{bong7}, for example). In addition, for some fixed $\theta>0$, we have the following inclusion relations (see \cite{tang})
\begin{equation*}
A^{\rho,\theta}_1(\mathbb R^d)\subset A^{\rho,\theta}_{p_1}(\mathbb R^d)\subset A^{\rho,\theta}_{p_2}(\mathbb R^d),
\end{equation*}
whenever $1\leq p_1<p_2<\infty$. As in \cite{tang} (see also \cite{bui} and \cite{wang}), we say that a weight $\omega$ is in the class
$A^{\rho,\theta}_{p,q}(\mathbb R^d)$ for $1<p,q<\infty$ and $0<\theta<\infty$, if there exists a positive constant $C>0$ such that
\begin{equation*}
\bigg(\frac{1}{|B|}\int_B \omega(x)^q\,dx\bigg)^{1/q}\bigg(\frac{1}{|B|}\int_B \omega(x)^{-{p'}}\,dx\bigg)^{1/{p'}}
\leq C\cdot\bigg(1+\frac{r}{\rho(x_0)}\bigg)^{\theta}
\end{equation*}
holds for all balls $B=B(x_0,r)\subset\mathbb R^d$. For the case $p=1$, we also say that a weight $\omega$ is in the class
$A^{\rho,\theta}_{1,q}(\mathbb R^d)$ for $1<q<\infty$ and $0<\theta<\infty$, if there exists a positive constant $C>0$ such that
\begin{equation*}
\bigg(\frac1{|B|}\int_B \omega(x)^q\,dx\bigg)^{1/q}\leq C\cdot\bigg(1+\frac{r}{\rho(x_0)}\bigg)^{\theta}
\underset{x\in B}{\mbox{ess\,inf}}\;\omega(x)
\end{equation*}
holds for all balls $B=B(x_0,r)\subset\mathbb R^d$. In view of \eqref{cc}, for any $1\leq p,q<\infty$, we find that
\begin{equation*}
A_{p,q}(\mathbb R^d)\subset A^{\rho,\theta_1}_{p,q}(\mathbb R^d)\subset A^{\rho,\theta_2}_{p,q}(\mathbb R^d),
\end{equation*}
whenever $0\leq\theta_1<\theta_2<\infty$. Here $A_{p,q}(\mathbb R^d)$ denotes the classical Muckenhoupt--Wheeden class (see \cite{muckenhoupt3}). Correspondingly, for $1\leq p,q<\infty$, we define
\begin{equation*}
A^{\rho,\infty}_{p,q}(\mathbb R^d):=\bigcup_{\theta>0}A^{\rho,\theta}_{p,q}(\mathbb R^d).
\end{equation*}
\begin{remark}
A few comments are in order:
\begin{enumerate}
\item As in the classical Muckenhoupt theory, we define the $A^{\rho,\theta}_{p}$ and $A^{\rho,\theta}_{p,q}$ characteristic constants of $\omega$ as follows ($B=B(x_0,r)$):
\begin{equation*}
\begin{split}
[\omega]_{A^{\rho,\theta}_p}
&:=\sup_{B\subset\mathbb R^d}\bigg(1+\frac{r}{\rho(x_0)}\bigg)^{-\theta}\bigg(\frac{1}{|B|}\int_B \omega(x)\,dx\bigg)^{1/p}\bigg(\frac{1}{|B|}\int_B \omega(x)^{-{p'}/p}\,dx\bigg)^{1/{p'}},~\mbox{when}~ 1<p<\infty,\\
[\omega]_{A^{\rho,\theta}_1}
&:=\sup_{B\subset\mathbb R^d}\bigg(1+\frac{r}{\rho(x_0)}\bigg)^{-\theta}\bigg(\frac{1}{|B|}\int_B \omega(x)\,dx\bigg)
\bigg(\underset{x\in B}{\mbox{ess\,inf}}\;\omega(x)\bigg)^{-1},~\mbox{when}~ p=1,\\
[\omega]_{A^{\rho,\theta}_{p,q}}
&:=\sup_{B\subset\mathbb R^d}\bigg(1+\frac{r}{\rho(x_0)}\bigg)^{-\theta}\bigg(\frac{1}{|B|}\int_B \omega(x)^q\,dx\bigg)^{1/q}\bigg(\frac{1}{|B|}\int_B \omega(x)^{-{p'}}\,dx\bigg)^{1/{p'}},~\mbox{when}~ 1<p,q<\infty,\\
[\omega]_{A^{\rho,\theta}_{1,q}}
&:=\sup_{B\subset\mathbb R^d}\bigg(1+\frac{r}{\rho(x_0)}\bigg)^{-\theta}\bigg(\frac1{|B|}\int_B \omega(x)^q\,dx\bigg)^{1/q}
\bigg(\underset{x\in B}{\mbox{ess\,inf}}\;\omega(x)\bigg)^{-1},~\mbox{when}~ p=1,1<q<\infty.
\end{split}
\end{equation*}
\item We mention that in the definition of both classes of weights $A^{\rho,\infty}_{p}(\mathbb R^d)$ and $A^{\rho,\infty}_{p,q}(\mathbb R^d)$ balls can be replaced by cubes, due to \eqref{com}.
\item For more results about weighted norm inequalities of various integral operators in harmonic analysis (such as first or second order Riesz--Schr\"{o}dinger transforms, Schr\"{o}dinger type singular integrals, fractional integrals, etc.), one can see \cite{bong1,bong2,bong4,bong5,bong6,bong7,bong8,bong9,bong10,bui,tang,tang2}.
\end{enumerate}
\end{remark}
In 2011, Bongioanni--Harboure--Salinas \cite{bong3} introduced a new class of function spaces (see also \cite{bong2}). According to \cite{bong3}, the new BMO space $\mathrm{BMO}_{\rho,\infty}(\mathbb R^d)$ is defined by
\begin{equation*}
\mathrm{BMO}_{\rho,\infty}(\mathbb R^d):=\bigcup_{\theta>0}\mathrm{BMO}_{\rho,\theta}(\mathbb R^d),
\end{equation*}
where for any fixed $0<\theta<\infty$ the space $\mathrm{BMO}_{\rho,\theta}(\mathbb R^d)$ is defined to be the set of all locally integrable functions $f$ satisfying
\begin{equation}\label{BM}
\frac{1}{|Q(x_0,r)|}\int_{Q(x_0,r)}\big|f(x)-f_{Q}\big|\,dx\leq C\cdot\bigg(1+\frac{r}{\rho(x_0)}\bigg)^{\theta},
\end{equation}
for all $x_0\in\mathbb R^d$ and $r\in(0,\infty)$, $f_{Q}$ denotes the mean value of $f$ on $Q(x_0,r)$, that is,
\begin{equation*}
f_{Q}:=\frac{1}{|Q(x_0,r)|}\int_{Q(x_0,r)}f(y)\,dy.
\end{equation*}
A norm for $f\in \mathrm{BMO}_{\rho,\theta}(\mathbb R^d)$, denoted by $\|f\|_{\mathrm{BMO}_{\rho,\theta}}$, is given by the infimum of the constants satisfying \eqref{BM}, after identifying functions that differ by a constant, or equivalently,
\begin{equation*}
\|f\|_{\mathrm{BMO}_{\rho,\theta}}
:=\sup_{Q(x_0,r)}\bigg(1+\frac{r}{\rho(x_0)}\bigg)^{-\theta}\bigg(\frac{1}{|Q(x_0,r)|}\int_{Q(x_0,r)}\big|f(x)-f_{Q}\big|\,dx\bigg),
\end{equation*}
where the supremum is taken over all cubes $Q(x_0,r)$ with $x_0\in\mathbb R^d$ and $r\in(0,\infty)$. Note that if we let $\theta=0$ in \eqref{BM}, we obtain the usual (John--Nirenberg) BMO space (see \cite{john}). Define
\begin{equation*}
\mathrm{BMO}_{\rho,\theta}(\mathbb R^d):=\Big\{f\in L^1_{\mathrm{loc}}(\mathbb R^d):\|f\|_{\mathrm{BMO}_{\rho,\theta}}<\infty\Big\}.
\end{equation*}
With the above definition in mind, one has
\begin{equation*}
\mathrm{BMO}(\mathbb R^d)\subset \mathrm{BMO}_{\rho,\theta_1}(\mathbb R^d)\subset \mathrm{BMO}_{\rho,\theta_2}(\mathbb R^d)
\end{equation*}
whenever $0<\theta_1<\theta_2<\infty$, and hence
\begin{equation*}
\mathrm{BMO}(\mathbb R^d)\subset\mathrm{BMO}_{\rho,\infty}(\mathbb R^d).
\end{equation*}
Moreover, it can be shown that the classical BMO space is properly contained in $\mathrm{BMO}_{\rho,\infty}(\mathbb R^d)$ (see \cite{bong2,bong3,tang} for more examples).

We give a version of John--Nirenberg inequality suitable for the new BMO space $\mathrm{BMO}_{\rho,\theta}(\mathbb R^d)$, which can be found in \cite[Proposition 4.2]{tang}.
\begin{lemma}[\cite{tang}]\label{expbmo}
If $f\in \mathrm{BMO}_{\rho,\theta}(\mathbb R^d)$ with $0<\theta<\infty$, then there exist two positive constants $C_1$ and $C_2$ such that for every cube $\mathcal{Q}=Q(x_0,r)$ and every $\lambda>0$,
\begin{equation*}
\Big|\Big\{x\in \mathcal{Q}:|f(x)-f_{\mathcal{Q}}|>\lambda\Big\}\Big|
\leq C_1|\mathcal{Q}|\exp\bigg\{-\bigg(1+\frac{r}{\rho(x_0)}\bigg)^{-(N_0+1)\theta}\frac{C_2\lambda}{\|f\|_{\mathrm{BMO}_{\rho,\theta}}}\bigg\},
\end{equation*}
where $N_0$ is the constant appearing in Lemma \ref{N0}.
\end{lemma}
This estimate plays an important role in the proofs of our main theorems.

In 2014, Liu--Sheng introduced a new class of function spaces which is larger than the classical Lipschitz space. According to \cite{liu}, for $0<\theta<\infty$ and $0\leq\beta<1$, the space $\mathrm{Lip}_{\beta}^{\rho,\theta}(\mathbb R^d)$ is defined to be the set of all locally integrable functions $f$ satisfying
\begin{equation}\label{Lipliu}
\frac{1}{|B(x_0,r)|^{1+\beta/d}}\int_{B(x_0,r)}\big|f(x)-f_{B}\big|\,dx
\leq C\cdot\bigg(1+\frac{r}{\rho(x_0)}\bigg)^{\theta},
\end{equation}
for all $x_0\in\mathbb R^d$ and $r\in(0,\infty)$, $f_{B}$ denotes the mean value of $f$ on $B(x_0,r)$, that is,
\begin{equation*}
f_{B}:=\frac{1}{|B(x_0,r)|}\int_{B(x_0,r)}f(y)\,dy.
\end{equation*}
The infimum of the constants $C$ satisfying \eqref{Lipliu} is defined to be the norm of $f\in \mathrm{Lip}_{\beta}^{\rho,\theta}(\mathbb R^d)$ and denoted by $\|f\|_{\mathrm{Lip}_{\beta}^{\rho,\theta}}$, or equivalently,
\begin{equation*}
\|f\|_{\mathrm{Lip}_{\beta}^{\rho,\theta}}:=\sup_{B(x_0,r)}\bigg(1+\frac{r}{\rho(x_0)}\bigg)^{-\theta}
\bigg(\frac{1}{|B(x_0,r)|^{1+\beta/d}}\int_{B(x_0,r)}\big|f(x)-f_{B}\big|\,dx\bigg),
\end{equation*}
where the supremum is taken over all balls $B(x_0,r)$ with $x_0\in\mathbb R^d$ and $r\in(0,\infty)$.
\begin{remark}
\begin{enumerate}
Some special cases:
  \item Note that if $\theta=0$ in \eqref{Lipliu}, then $\mathrm{Lip}_{\beta}^{\rho,\theta}(\mathbb R^d)$ is exactly the classical Lipschitz space $\mathrm{Lip}_{\beta}(\mathbb R^d)$;
  \item if $\beta=0$ and $0<\theta<\infty$ in \eqref{Lipliu}, then $\mathrm{Lip}_{\beta}^{\rho,\theta}(\mathbb R^d)$ is exactly the above space $\mathrm{BMO}_{\rho,\theta}(\mathbb R^d)$ introduced by Bongioanni--Harboure--Salinas in \cite{bong3}.
\end{enumerate}
\end{remark}

For such spaces, we have the following key estimate, which can be found in \cite[Theorem 5]{liu}.
\begin{lemma}[\cite{liu}]\label{beta}
If $f\in \mathrm{Lip}_{\beta}^{\rho,\theta}(\mathbb R^d)$ with $0<\beta<1$ and $0<\theta<\infty$, then there exists a positive constant $C>0$ such that
\begin{equation*}
\frac{|f(x)-f(y)|}{|x-y|^{\beta}}\leq C\|f\|_{\mathrm{Lip}_{\beta}^{\rho,\theta}}
\bigg(1+\frac{|x-y|}{\rho(x)}+\frac{|x-y|}{\rho(y)}\bigg)^{\theta}
\end{equation*}
holds true for all $x,y\in\mathbb R^d$ with $x\neq y$. Conversely, if there is a positive constant $C>0$ such that for any $x,y\in\mathbb R^d$ with $x\neq y$,
\begin{equation*}
\frac{|f(x)-f(y)|}{|x-y|^{\beta}}\leq C\bigg(1+\frac{|x-y|}{\rho(x)}+\frac{|x-y|}{\rho(y)}\bigg)^{\theta}
\end{equation*}
holds for some $\theta>0$ and $0<\beta<1$, then $f\in \mathrm{Lip}_{\beta}^{\rho,(N_0+1)\theta}(\mathbb R^d)$.
\end{lemma}
Throughout this article, we will always assume that $V\in RH_s(\mathbb R^d)$ for some $s\geq d/2$ and $d\geq3$, the letter $C$ denotes a positive constant which is independent of the main parameters, but it may vary from line to line. Constants with subscripts, such as $C_0,C_1$, do not change in different occurrences. The notation $\mathbf{X}\approx \mathbf{Y}$ means that there exists a positive constant $C>0$ such that $1/C\leq \mathbf{X}/\mathbf{Y}\leq C$.

\section{\textbf{Some lemmas}}
Let us recall and prove some lemmas, before stating and giving the proof of our main theorems. First observe that from \eqref{com}, it is easy to verify that when $x\in B(x_0,r)$ with $x_0\in\mathbb R^d$ and $r>0$,
\begin{equation}\label{wangh3}
1+\frac{r}{\rho(x)}\leq C_0\cdot\bigg(1+\frac{r}{\rho(x_0)}\bigg)^{N_0+1},
\end{equation}
where $C_0$ is the constant appearing in Lemma \ref{N0}. In fact, this estimate has been obtained in the literature (see
\cite[Lemma 1]{bong8} and \cite[Lemma 2]{bong10}), for the sake of completeness, we give its proof here. By the left-hand side of \eqref{com}, we have that for any $x\in B(x_0,r)$,
\begin{equation*}
\frac{1}{\rho(x)}\leq C_0\cdot\frac{1}{\rho(x_0)}\bigg(1+\frac{|x-x_0|}{\rho(x_0)}\bigg)^{N_0}
<C_0\cdot\frac{1}{\rho(x_0)}\bigg(1+\frac{r}{\rho(x_0)}\bigg)^{N_0}.
\end{equation*}
From this, it follows that (since $C_0\geq1$)
\begin{equation*}
1+\frac{r}{\rho(x)}\leq 1+C_0\cdot\frac{r}{\rho(x_0)}\bigg(1+\frac{r}{\rho(x_0)}\bigg)^{N_0}
\leq C_0\cdot\bigg(1+\frac{r}{\rho(x_0)}\bigg)^{N_0+1},
\end{equation*}
as desired. The following results (Lemmas \ref{rh} through \ref{Apq}) are extensions of well-known properties of classical $A_p$ weights. We first present an important property of the $A^{\rho,\theta}_p$ classes of weights with $1\leq p<\infty$ and $0<\theta<\infty$, which was given by Bongioanni--Harboure--Salinas in \cite[Lemma 5]{bong1}.

\begin{lemma}[\cite{bong1}]\label{rh}
If $\omega\in A^{\rho,\theta}_p(\mathbb R^d)$ with $0<\theta<\infty$ and $1\leq p<\infty$, then there exist positive constants $\epsilon>0,\eta>1$ and $C>0$ such that
\begin{equation}\label{rholder}
\bigg(\frac{1}{|\mathcal{Q}|}\int_{\mathcal{Q}}\omega(x)^{1+\epsilon}dx\bigg)^{\frac{1}{1+\epsilon}}
\leq C\cdot\bigg(\frac{1}{|\mathcal{Q}|}\int_{\mathcal{Q}}\omega(x)\,dx\bigg)\bigg(1+\frac{r}{\rho(x_0)}\bigg)^{\eta}
\end{equation}
holds for every cube $\mathcal{Q}=Q(x_0,r)$ in $\mathbb R^d$.
\end{lemma}
\begin{remark}
The constant $C>0$ in Lemma \ref{rh} depends on $p,d$ and the $A^{\rho,\theta}_{p}$ characteristic constant of $\omega$, the positive number $\epsilon$ in Lemma \ref{rh} comes from the classical proof for $A_p$ weights in \cite[Theorem 7.4]{duoand}, and $\eta$ is a positive constant greater than 1, which can be chosen as follows.
\begin{equation*}
\eta:=\theta p+(\theta+d)\frac{pN_0}{N_0+1}+(N_0+1)\frac{d\epsilon}{1+\epsilon}>1.
\end{equation*}
One is naturally led to ask whether it is possible to improve this result.
\end{remark}
As a direct consequence of Lemma \ref{rh}, we have the following result, which provides us the comparison between the Lebesgue measure of a set $E$ and its weighted measure $\omega(E)$.
\begin{lemma}\label{comparelem}
If $\omega\in A^{\rho,\theta}_p(\mathbb R^d)$ with $0<\theta<\infty$ and $1\leq p<\infty$, then there exist two positive numbers $\delta>0$ and $\eta>1$ such that for any cube $\mathcal{Q}=Q(x_0,r)\subset\mathbb R^d$,
\begin{equation}\label{compare}
\frac{\omega(E)}{\omega(\mathcal{Q})}\leq C\cdot\bigg(\frac{|E|}{|\mathcal{Q}|}\bigg)^\delta\bigg(1+\frac{r}{\rho(x_0)}\bigg)^{\eta}
\end{equation}
holds for any measurable subset $E$ contained in $\mathcal{Q}$, where $C>0$ is a constant which does not depend on $E$ nor on $\mathcal{Q}$, and $\eta$ is given as in Lemma \ref{rh}.
\end{lemma}
\begin{proof}
For any given cube $\mathcal{Q}=Q(x_0,r)$ with $x_0\in\mathbb R^d$ and $r\in(0,\infty)$, suppose that $E\subset \mathcal{Q}$, then by H\"older's inequality with exponent $1+\epsilon$ and \eqref{rholder}, we can deduce that
\begin{equation*}
\begin{split}
\omega(E)&=\int_{\mathcal{Q}}\chi_E(x)\cdot \omega(x)\,dx\\
&\leq\bigg(\int_{\mathcal{Q}} \omega(x)^{1+\epsilon}dx\bigg)^{\frac{1}{1+\epsilon}}
\bigg(\int_{\mathcal{Q}}\chi_E(x)^{\frac{1+\epsilon}{\epsilon}}\,dx\bigg)^{\frac{\epsilon}{1+\epsilon}}\\
&\leq C\cdot|\mathcal{Q}|^{\frac{1}{1+\epsilon}}\bigg(\frac{1}{|\mathcal{Q}|}\int_{\mathcal{Q}}\omega(x)\,dx\bigg)
\bigg(1+\frac{r}{\rho(x_0)}\bigg)^{\eta}|E|^{\frac{\epsilon}{1+\epsilon}}\\
&=C\cdot\bigg(\frac{|E|}{|\mathcal{Q}|}\bigg)^{\frac{\epsilon}{1+\epsilon}}\bigg(1+\frac{r}{\rho(x_0)}\bigg)^{\eta}\omega(\mathcal{Q}).
\end{split}
\end{equation*}
This gives \eqref{compare} with $\delta=\epsilon/{(1+\epsilon)}$. Here the characteristic function of the set $E$ is denoted by $\chi_E$.
\end{proof}

The following result gives the relationship between these two classes of weights, $A^{\rho,\infty}_p(\mathbb R^d)$ and $A^{\rho,\infty}_{p,q}(\mathbb R^d)$.
\begin{lemma}\label{Apq}
Suppose that $1\leq p<q<\infty$. Then the following statements are true.
\begin{enumerate}
  \item If $p>1$ and $0<\theta<\infty$, then $\omega\in A^{\rho,\theta}_{p,q}(\mathbb R^d)$ implies that $\omega^q\in A^{\rho,\widetilde{\theta}}_t(\mathbb R^d)$ with
\begin{equation*}
t:=1+q/{p'}\quad and \quad \widetilde{\theta}:=\theta\cdot\frac{1}{1/q+1/{p'}}.
\end{equation*}
  \item If $p=1$ and $0<\theta<\infty$, then $\omega\in A^{\rho,\theta}_{1,q}(\mathbb R^d)$ implies that $\omega^q\in A^{\rho,\theta^{\ast}}_1(\mathbb R^d)$ with
  \begin{equation*}
  \theta^{\ast}:=\theta\cdot q.
  \end{equation*}
\end{enumerate}
\end{lemma}

\begin{proof}
(1) When $t=1+q/{p'}$, then a simple computation shows that
\begin{equation*}
\frac{\,1\,}{t}=\frac{\,1\,}{q}\cdot\frac{1}{1/q+1/{p'}},\qquad \frac{\,1\,}{t'}=\frac{t-1}{t}=\frac{\,1\,}{p'}\cdot\frac{1}{1/q+1/{p'}},
\end{equation*}
and
\begin{equation*}
q\cdot\Big(-\frac{t'}{\,t\,}\Big)=-q\cdot\frac{1}{t-1}=-p'.
\end{equation*}
If $\omega\in A^{\rho,\theta}_{p,q}(\mathbb R^d)$ with $1<p<q<\infty$ and $0<\theta<\infty$, then we have
\begin{equation*}
\begin{split}
&\bigg(\frac{1}{|B|}\int_B \omega^q(x)\,dx\bigg)^{1/t}\bigg(\frac{1}{|B|}\int_B \omega^q(x)^{-{t'}/t}\,dx\bigg)^{1/{t'}}\\
&=\bigg[\bigg(\frac{1}{|B|}\int_B \omega(x)^q\,dx\bigg)^{1/q}\bigg(\frac{1}{|B|}\int_B \omega(x)^{-{p'}}\,dx\bigg)^{1/{p'}}\bigg]^{\frac{1}{1/q+1/{p'}}}\\
&\leq\Big([\omega]_{A^{\rho,\theta}_{p,q}}\Big)^{\frac{1}{1/q+1/{p'}}}
\cdot\bigg(1+\frac{r}{\rho(x_0)}\bigg)^{\theta\cdot\frac{1}{1/q+1/{p'}}}
=\Big([\omega]_{A^{\rho,\theta}_{p,q}}\Big)^{\frac{1}{1/q+1/{p'}}}\cdot\bigg(1+\frac{r}{\rho(x_0)}\bigg)^{\widetilde{\theta}},
\end{split}
\end{equation*}
which means that $\omega^q\in A^{\rho,\widetilde{\theta}}_{t}(\mathbb R^d)$ with $\widetilde{\theta}=\theta\cdot\frac{1}{1/q+1/{p'}}$, and
\begin{equation*}
[\omega^q]_{A^{\rho,\widetilde{\theta}}_{t}}\leq\Big([\omega]_{A^{\rho,\theta}_{p,q}}\Big)^{\frac{1}{1/q+1/{p'}}}
=\Big([\omega]_{A^{\rho,\theta}_{p,q}}\Big)^{{\widetilde{\theta}}/{\theta}}.
\end{equation*}
Here and in the sequel, for any positive number $\gamma>0$, we denote $\omega^{\gamma}(x)=\omega(x)^{\gamma}$ by convention.

(2) On the other hand, if $\omega\in A^{\rho,\theta}_{1,q}(\mathbb R^d)$ with $1<q<\infty$ and $0<\theta<\infty$, then we have
\begin{equation*}
\begin{split}
\frac1{|B|}\int_B \omega^q(x)\,dx
&\leq \Big([\omega]_{A^{\rho,\theta}_{1,q}}\Big)^q\cdot\bigg(1+\frac{r}{\rho(x_0)}\bigg)^{\theta\cdot q}\Big(\underset{x\in B}{\mbox{ess\,inf}}\;\omega(x)\Big)^q\\
&=\Big([\omega]_{A^{\rho,\theta}_{1,q}}\Big)^q
\cdot\bigg(1+\frac{r}{\rho(x_0)}\bigg)^{\theta^{*}}\underset{x\in B}{\mbox{ess\,inf}}\;\omega^q(x),
\end{split}
\end{equation*}
which means that $\omega^q\in A^{\rho,\theta^{*}}_{1}(\mathbb R^d)$ with $\theta^{*}=\theta\cdot q$, and
\begin{equation*}
[\omega^q]_{A^{\rho,\theta^{*}}_{1}}\leq\Big([\omega]_{A^{\rho,\theta}_{1,q}}\Big)^q
=\Big([\omega]_{A^{\rho,\theta}_{1,q}}\Big)^{{\theta^{*}}/{\theta}}.
\end{equation*}
This concludes the proof of Lemma \ref{Apq}.
\end{proof}
There are many classical works about the characterizations of usual BMO and Lipschitz spaces. We now present some relevant results in the literature. In 1961, John and Nirenberg established the following result (known as the John--Nirenberg inequality, see \cite{john} and \cite{duoand}):
If $f\in \mathrm{BMO}(\mathbb R^d)$, then for any cube $Q$ in $\mathbb R^d$ and for any $\lambda>0$,
\begin{equation*}
\Big|\Big\{x\in Q:|f(x)-f_{Q}|>\lambda\Big\}\Big|
\leq c_1|Q|\exp\bigg\{-\frac{c_2\lambda}{\|f\|_{\mathrm{BMO}}}\bigg\},
\end{equation*}
where $c_1>0$ and $c_2>0$ are two universal constants and
\begin{equation*}
\|f\|_{\mathrm{BMO}}:=\sup_{Q\subset\mathbb R^d}\frac{1}{|Q|}\int_{Q}|f(x)-f_{Q}|\,dx<\infty.
\end{equation*}
As a consequence of the above estimate and H\"{o}lder's inequality, we can obtain an equivalent norm on $\mathrm{BMO}(\mathbb R^d)$, see \cite[Corollary 6.12]{duoand}, for example.
\begin{proposition}[\cite{duoand}]
For $1\leq s<\infty$, define
\begin{equation*}
\|f\|_{\mathrm{BMO}^s}:=\sup_{Q\subset\mathbb R^d}\bigg(\frac{1}{|Q|}\int_{Q}|f(x)-f_{Q}|^s\,dx\bigg)^{1/s}.
\end{equation*}
Then we have~(when $s=1$, we write $\|\cdot\|_{\mathrm{BMO}^s}=\|\cdot\|_{\mathrm{BMO}}$)
\begin{equation*}
\|f\|_{\mathrm{BMO}^s}\approx\|f\|_{\mathrm{BMO}},
\end{equation*}
for each $1<s<\infty$ (the norms are mutually equivalent).
\end{proposition}
We can extend this result to the weighted case. For each $\omega\in A_{\infty}:=\cup_{1\leq p<\infty}A_p$, we denote by $\mathrm{BMO}({\omega})$ the set of all locally integrable functions $f$ on $\mathbb R^d$ such that
\begin{equation*}
\|f\|_{\mathrm{BMO}(\omega)}:=\sup_{Q\subset\mathbb R^d}\frac{1}{\omega(Q)}\int_{Q}|f(x)-f_{\omega,Q}|\omega(x)\,dx<\infty,
\end{equation*}
where
\begin{equation*}
f_{\omega,Q}:=\frac{1}{\omega(Q)}\int_{Q}f(x)\omega(x)\,dx.
\end{equation*}
In 1976, Muckenhoupt and Wheeden proved that a function $f$ is in the space $\mathrm{BMO}(\mathbb R^d)$ if and only if $f$ is in $\mathrm{BMO}({\omega})$ (bounded mean oscillation with respect to $\omega$), provided that $\omega\in A_{\infty}(\mathbb R^d)$, see \cite[Theorem 5]{muck2}.
\begin{proposition}[\cite{muck2}]
For each $\omega\in A_{\infty}(\mathbb R^d)$, then we have $\mathrm{BMO}(\mathbb R^d)=\mathrm{BMO}({\omega})$ and (the norms are mutually equivalent)
\begin{equation*}
\|f\|_{\mathrm{BMO}(\omega)}\approx\|f\|_{\mathrm{BMO}}.
\end{equation*}
\end{proposition}
In 2011, Ho further proved the following result by using the John--Nirenberg inequality and relevant properties of $A_{\infty}$ weights, see \cite[Theorem 3.1]{ho}.
\begin{proposition}[\cite{ho}]
For all $1\leq s<\infty$ and $\omega\in A_{s}(\mathbb R^d)$, then $f\in \mathrm{BMO}(\mathbb R^d)$ if and only if
\begin{equation*}
\sup_{Q\subset\mathbb R^d}\bigg(\frac{1}{\omega(Q)}\int_{Q}|f(x)-f_{Q}|^s\omega(x)\,dx\bigg)^{1/s}<\infty.
\end{equation*}
\end{proposition}
On the other hand, it is well known that Lipschitz spaces are useful tools in the regularity theory of PDEs. We have the following characterization
of classical Lipschitz spaces, which can be found in \cite[Lemma 1.5]{pal} and \cite[Theorem 2]{janson}. For more general results, see \cite[Theorem 2.4]{nakai}.
\begin{proposition}[\cite{pal}]
For $1\leq s<\infty$ and $0<\beta<1$, define
\begin{equation*}
\|f\|_{\mathrm{Lip}_{\beta}^s}:=\sup_{B\subset\mathbb R^d}\frac{1}{|B|^{\beta/n}}\bigg(\frac{1}{|B|}\int_{B}|f(x)-f_{B}|^s\,dx\bigg)^{1/s},
\end{equation*}
and
\begin{equation*}
\|f\|_{\Lambda_{\beta}}:=\sup_{x,y\in \mathbb R^d,x\neq y}\frac{|f(x)-f(y)|}{|x-y|^{\beta}}.
\end{equation*}
Then we have (when $s=1$, we denote $\|\cdot\|_{\mathrm{Lip}_{\beta}^s}=\|\cdot\|_{\mathrm{Lip}_{\beta}}$)
\begin{equation*}
\|f\|_{\mathrm{Lip}_{\beta}^s}\approx\|f\|_{\mathrm{Lip}_{\beta}}\approx\|f\|_{\Lambda_{\beta}},
\end{equation*}
for each $1<s<\infty$.
\end{proposition}
We mention that this result leads to a generalization of the classical Sobolev embedding theorem. It is also well known that $\mathrm{Lip}_{1/p-1}(\mathbb R^d)$ is the dual space of Hardy space $H^p(\mathbb R^d)$ when $0<p<1$, and $\mathrm{Lip}_{0}(\mathbb R^d)=\mathrm{BMO}(\mathbb R^d)$ is the dual space of Hardy space $H^1(\mathbb R^d)$.
\begin{remark}
There are some other characterizations of Lipschitz spaces, which have been obtained by several authors. For instance, we can give some new characterizations of Lipschitz spaces via the boundedness of commutators (such as Calder\'{o}n--Zygmund singular integral operators and fractional integrals). For further details, we refer the reader to \cite{deng,duong,pal,shi,shi2} and the references therein.
\end{remark}
It is natural to consider the same problems (characterizations of function spaces) in the Schr\"{o}dinger context. Concerning the BMO and Lipschitz spaces related to Schr\"{o}dinger operators with nonnegative potentials, we can obtain the following conclusions.
\begin{proposition}
Let $0<\theta<\infty$ and $1\leq s<\infty$. If $f\in \mathrm{BMO}_{\rho,\theta}(\mathbb R^d)$, then there exists a positive constant $C>0$ such that
\begin{equation*}
\bigg(\frac{1}{|\mathcal{Q}|}\int_{\mathcal{Q}}|f(x)-f_{\mathcal{Q}}|^s\,dx\bigg)^{1/s}
\leq C\bigg(1+\frac{r}{\rho(x_0)}\bigg)^{(N_0+1)\theta}\|f\|_{\mathrm{BMO}_{\rho,\theta}}
\end{equation*}
holds for every cube $\mathcal{Q}=Q(x_0,r)$ with $x_0\in\mathbb R^d$ and $r>0$, where $N_0$ is the constant appearing in Lemma $\ref{N0}$.
\end{proposition}
This result was first proved by Bongioanni--Harboure--Salinas in 2011, see \cite[Proposition 3]{bong3}.
\begin{proposition}
Let $0<\theta<\infty$ and $1\leq s<\infty$. If $f\in \mathrm{Lip}_{\beta}^{\rho,\theta}(\mathbb R^d)$ with $0<\beta<1$, then there exists a positive constant $C>0$ such that
\begin{equation*}
\frac{1}{|{\mathcal{B}}|^{\beta/d}}\bigg(\frac{1}{|\mathcal{B}|}\int_{\mathcal{B}}|f(x)-f_{\mathcal{B}}|^s\,dx\bigg)^{1/s}
\leq C\bigg(1+\frac{r}{\rho(x_0)}\bigg)^{(N_0+1)\theta}\|f\|_{\mathrm{Lip}_{\beta}^{\rho,\theta}}
\end{equation*}
holds for every ball $\mathcal{B}=B(x_0,r)$ with $x_0\in\mathbb R^d$ and $r>0$, where $N_0$ is the constant appearing in Lemma $\ref{N0}$.
\end{proposition}
This result was first given by Liu--Sheng in 2014, see \cite[Proposition 3]{liu}.

Inspired by these results, it is natural to ask whether similar result remains true in the weighted case. In this paper we give a positive answer to this question. As already mentioned in the introduction, the harmonic analysis arising from the Schr\"{o}dinger operator $\mathcal{L}=-\Delta+V$ is based on the use of a related critical radius function, which was introduced by Shen. In this framework, to show our main results, we rely on a version of John--Nirenberg's inequality for the space $\mathrm{BMO}_{\rho,\theta}(\mathbb R^d)$(see Lemma \ref{expbmo}), a pointwise estimate for the function $f\in \mathrm{Lip}_{\beta}^{\rho,\theta}(\mathbb R^d)$(see Lemma \ref{beta}), and some related properties of classes of weights.

\section{\textbf{Main results}}
Let $N_0$ be the same constant as in Lemma \ref{N0} and let $\eta$ be the same number as in Lemma \ref{comparelem}. Now we are in a position to give the main results of this paper.
\begin{theorem}\label{main1}
Let $1\leq p<\infty$ and $\omega\in A^{\rho,\theta_2}_p(\mathbb R^d)$ with $0<\theta_2<\infty$. Then the following
statements are true.
\begin{enumerate}
\item If $f\in \mathrm{BMO}_{\rho,\theta_1}(\mathbb R^d)$ with $0<\theta_1<\infty$, then for any cube $\mathcal{Q}=Q(x_0,r)\subset\mathbb R^d$,
\begin{equation*}
\bigg(\frac{1}{\omega(\mathcal{Q})}\int_{\mathcal{Q}}|f(x)-f_{\mathcal{Q}}|^p\omega(x)\,dx\bigg)^{1/p}
\leq C\bigg(1+\frac{r}{\rho(x_0)}\bigg)^{(N_0+1)\theta_1+\eta/p}\|f\|_{\mathrm{BMO}_{\rho,\theta_1}}.
\end{equation*}
\item Conversely, if there exists a constant $C>0$ such that for any cube $\mathcal{Q}=Q(x_0,r)\subset\mathbb R^d$,
\begin{equation}\label{assum1}
\bigg(\frac{1}{\omega(\mathcal{Q})}\int_{\mathcal{Q}}|f(x)-f_{\mathcal{Q}}|^p\omega(x)\,dx\bigg)^{1/p}\leq C\bigg(1+\frac{r}{\rho(x_0)}\bigg)^{\theta_1-\theta_2}
\end{equation}
holds for some $\theta_1>0$, then $f\in\mathrm{BMO}_{\rho,\theta_1}(\mathbb R^d)$, and
\begin{equation*}
\|f\|_{\mathrm{BMO}_{\rho,\theta_1}}\leq C[\omega]_{A^{\rho,\theta_2}_p}.
\end{equation*}
\end{enumerate}
\end{theorem}

\begin{theorem}\label{main2}
Let $1\leq p<q<\infty$ and $\omega\in A^{\rho,\theta_2}_{p,q}(\mathbb R^d)$ with $0<\theta_2<\infty$. Then the following
statements are true.
\begin{enumerate}
\item If $f\in \mathrm{BMO}_{\rho,\theta_1}(\mathbb R^d)$ with $0<\theta_1<\infty$, then for any cube $\mathcal{Q}=Q(x_0,r)\subset\mathbb R^d$,
\begin{equation*}
\bigg(\frac{1}{\omega^q(\mathcal{Q})}\int_{\mathcal{Q}}|f(x)-f_{\mathcal{Q}}|^q\omega(x)^q\,dx\bigg)^{1/q}
\leq C\bigg(1+\frac{r}{\rho(x_0)}\bigg)^{(N_0+1)\theta_1+\eta/q}\|f\|_{\mathrm{BMO}_{\rho,\theta_1}}.
\end{equation*}
\item Conversely, if there exists a constant $C>0$ such that for any cube $\mathcal{Q}=Q(x_0,r)\subset\mathbb R^d$,
\begin{equation}\label{assum2}
\bigg(\frac{1}{\omega^q(\mathcal{Q})}\int_{\mathcal{Q}}|f(x)-f_{\mathcal{Q}}|^q\omega(x)^q\,dx\bigg)^{1/q}\leq C\bigg(1+\frac{r}{\rho(x_0)}\bigg)^{\theta_1-\theta_2}
\end{equation}
holds for some $\theta_1>0$, then $f\in \mathrm{BMO}_{\rho,\theta_1}(\mathbb R^d)$, and
\begin{equation*}
\|f\|_{\mathrm{BMO}_{\rho,\theta_1}}\leq C[\omega]_{A^{\rho,\theta_2}_{p,q}}.
\end{equation*}
\end{enumerate}
\end{theorem}

\begin{proof}[Proof of Theorem \ref{main1}]
(1) Let $f\in \mathrm{BMO}_{\rho,\theta_1}(\mathbb R^d)$ with $0<\theta_1<\infty$. According to Lemma \ref{expbmo}, there are two constants $C_1,C_2>0$ such that for any $\lambda>0$,
\begin{equation*}
\Big|\Big\{x\in \mathcal{Q}:|f(x)-f_{\mathcal{Q}}|>\lambda\Big\}\Big|
\leq C_1|\mathcal{Q}|\exp\bigg\{-\bigg(1+\frac{r}{\rho(x_0)}\bigg)^{-(N_0+1)\theta_1}\frac{C_2\lambda}{\|f\|_{\mathrm{BMO}_{\rho,\theta_1}}}\bigg\}.
\end{equation*}
Since $\omega\in A^{\rho,\theta_2}_p(\mathbb R^d)$ with $0<\theta_2<\infty$ and $1\leq p<\infty$, by using Lemma \ref{comparelem}, we get
\begin{equation*}
\begin{split}
\omega\Big(\Big\{x\in \mathcal{Q}:|f(x)-f_{\mathcal{Q}}|>\lambda\Big\}\Big)
&\leq C\cdot C_1^{\delta}\omega(\mathcal{Q})
\exp\bigg\{-\bigg(1+\frac{r}{\rho(x_0)}\bigg)^{-(N_0+1)\theta_1}\frac{C_2\delta\lambda}{\|f\|_{\mathrm{BMO}_{\rho,\theta_1}}}\bigg\}\\
&\times\bigg(1+\frac{r}{\rho(x_0)}\bigg)^{\eta}.
\end{split}
\end{equation*}
Hence, for any cube $\mathcal{Q}\subset\mathbb R^d$,
\begin{align*}
&\bigg(\frac{1}{\omega(\mathcal{Q})}\int_{\mathcal{Q}}|f(x)-f_{\mathcal{Q}}|^p\omega(x)\,dx\bigg)^{1/p}\nonumber\\
&=\bigg(\frac{1}{\omega(\mathcal{Q})}\int_0^{\infty}p\lambda^{p-1}\omega\Big(\Big\{x\in \mathcal{Q}:|f(x)-f_{\mathcal{Q}}|>\lambda\Big\}\Big)\,d\lambda\bigg)^{1/p}\nonumber\\
&\leq\bigg(C\cdot C_1^{\delta}\int_0^{\infty}p\lambda^{p-1}\exp\bigg\{-\bigg(1+\frac{r}{\rho(x_0)}\bigg)^{-(N_0+1)\theta_1}
\frac{C_2\delta\lambda}{\|f\|_{\mathrm{BMO}_{\rho,\theta_1}}}\bigg\}d\lambda\bigg)^{1/p}\nonumber\\
&\times\bigg(1+\frac{r}{\rho(x_0)}\bigg)^{\eta/p}.\nonumber\\
\end{align*}
By making the substitution $\mu=\big(1+\frac{r}{\rho(x_0)}\big)^{-(N_0+1)\theta_1}\frac{C_2\delta\lambda}{\|f\|_{\mathrm{BMO}_{\rho,\theta_1}}}$, we can deduce that
\begin{align}\label{11}
&\bigg(\frac{1}{\omega(\mathcal{Q})}\int_{\mathcal{Q}}|f(x)-f_{\mathcal{Q}}|^p\omega(x)\,dx\bigg)^{1/p}\nonumber\\
&\leq\Big(C\cdot C_1^{\delta}p\Big)^{1/p}
\bigg(1+\frac{r}{\rho(x_0)}\bigg)^{(N_0+1)\theta_1}\frac{\|f\|_{\mathrm{BMO}_{\rho,\theta_1}}}{C_2\delta}
\times\bigg(1+\frac{r}{\rho(x_0)}\bigg)^{\eta/p}\nonumber\\
&\times\bigg(\int_0^{\infty}\mu^{p-1}e^{-\mu}\,d\mu\bigg)^{1/p}\nonumber\\
&\leq C\cdot\frac{C_1^{\delta/p}}{C_2}\bigg(1+\frac{r}{\rho(x_0)}\bigg)^{(N_0+1)\theta_1+\eta/p}\|f\|_{\mathrm{BMO}_{\rho,\theta_1}}.
\end{align}
This gives the desired inequality. Now we prove $(2)$. When $1<p<\infty$, by using H\"older's inequality, the condition $\omega\in A^{\rho,\theta_2}_p(\mathbb R^d)$ and \eqref{assum1}, we obtain
\begin{align}\label{12}
\frac{1}{|\mathcal{Q}|}\int_{\mathcal{Q}}|f(x)-f_{\mathcal{Q}}|\,dx
&=\frac{1}{|\mathcal{Q}|}\int_{\mathcal{Q}}|f(x)-f_{\mathcal{Q}}|\omega(x)^{1/p}\cdot\omega(x)^{-1/p}\,dx\nonumber\\
&\leq\frac{1}{|\mathcal{Q}|}\bigg(\int_{\mathcal{Q}}|f(x)-f_{\mathcal{Q}}|^p\omega(x)\,dx\bigg)^{1/p}
\bigg(\int_{\mathcal{Q}}\omega(x)^{-{p'}/p}\,dx\bigg)^{1/{p'}}\nonumber\\
&\leq C\bigg(1+\frac{r}{\rho(x_0)}\bigg)^{\theta_1-\theta_2}\nonumber\\
&\times\bigg(\frac{1}{|\mathcal{Q}|}\int_{\mathcal{Q}}\omega(x)\,dx\bigg)^{1/p}\bigg(\frac{1}{|\mathcal{Q}|}\int_{\mathcal{Q}}\omega(x)^{-{p'}/p}\,dx\bigg)^{1/{p'}}\nonumber\\
&\leq C[\omega]_{A^{\rho,\theta_2}_p}\bigg(1+\frac{r}{\rho(x_0)}\bigg)^{\theta_1}.
\end{align}
When $p=1$, then it follows immediately from the condition $\omega\in A^{\rho,\theta_2}_1(\mathbb R^d)$ and \eqref{assum1} that
\begin{align}\label{13}
\frac{1}{|\mathcal{Q}|}\int_{\mathcal{Q}}|f(x)-f_{\mathcal{Q}}|\,dx
&=\frac{1}{|\mathcal{Q}|}\int_{\mathcal{Q}}|f(x)-f_{\mathcal{Q}}|\omega(x)\cdot\omega(x)^{-1}\,dx\nonumber\\
&\leq\frac{1}{|\mathcal{Q}|}\bigg(\int_{\mathcal{Q}}|f(x)-f_{\mathcal{Q}}|\omega(x)\,dx\bigg)
\bigg(\underset{x\in \mathcal{Q}}{\mbox{ess\,sup}}\;\omega(x)^{-1}\bigg)\nonumber\\
&\leq C\bigg(1+\frac{r}{\rho(x_0)}\bigg)^{\theta_1-\theta_2}\nonumber\\
&\times\bigg(\frac{1}{|\mathcal{Q}|}\int_{\mathcal{Q}}\omega(x)\,dx\bigg)\bigg(\underset{x\in \mathcal{Q}}{\mbox{ess\,inf}}\;\omega(x)\bigg)^{-1}\nonumber\\
&\leq C[\omega]_{A^{\rho,\theta_2}_1}\bigg(1+\frac{r}{\rho(x_0)}\bigg)^{\theta_1}.
\end{align}
Collecting the above estimates \eqref{12} and \eqref{13}, we conclude the proof of Theorem \ref{main1}.
\end{proof}

\begin{proof}[Proof of Theorem \ref{main2}]
(1) Let $f\in \mathrm{BMO}_{\rho,\theta_1}(\mathbb R^d)$ with $0<\theta_1<\infty$. According to Lemma \ref{expbmo}, there are two constants $C_1,C_2>0$ such that for any $\lambda>0$,
\begin{equation*}
\Big|\Big\{x\in \mathcal{Q}:|f(x)-f_{\mathcal{Q}}|>\lambda\Big\}\Big|
\leq C_1|\mathcal{Q}|\exp\bigg\{-\bigg(1+\frac{r}{\rho(x_0)}\bigg)^{-(N_0+1)\theta_1}\frac{C_2\lambda}{\|f\|_{\mathrm{BMO}_{\rho,\theta_1}}}\bigg\}.
\end{equation*}
Since $\omega\in A^{\rho,\theta_2}_{p,q}(\mathbb R^d)$ with $0<\theta_2<\infty$ and $1\leq p<q<\infty$, by using Lemma \ref{Apq} and Lemma \ref{comparelem}, we have
\begin{equation*}
\begin{split}
\omega^q\Big(\Big\{x\in \mathcal{Q}:|f(x)-f_{\mathcal{Q}}|>\lambda\Big\}\Big)
&\leq C\cdot C_1^{\delta}\omega^q(\mathcal{Q})
\exp\bigg\{-\bigg(1+\frac{r}{\rho(x_0)}\bigg)^{-(N_0+1)\theta_1}\frac{C_2\delta\lambda}{\|f\|_{\mathrm{BMO}_{\rho,\theta_1}}}\bigg\}\\
&\times\bigg(1+\frac{r}{\rho(x_0)}\bigg)^{\eta}.
\end{split}
\end{equation*}
Hence, for any cube $\mathcal{Q}\subset\mathbb R^d$,
\begin{align*}
&\bigg(\frac{1}{\omega^q(\mathcal{Q})}\int_{\mathcal{Q}}|f(x)-f_{\mathcal{Q}}|^q\omega(x)^q\,dx\bigg)^{1/q}\nonumber\\
&=\bigg(\frac{1}{\omega^q(\mathcal{Q})}\int_0^{\infty}q\lambda^{q-1}\omega^q\Big(\Big\{x\in \mathcal{Q}:|f(x)-f_{\mathcal{Q}}|>\lambda\Big\}\Big)\,d\lambda\bigg)^{1/q}\nonumber\\
&\leq\bigg(C\cdot C_1^{\delta}\int_0^{\infty}q\lambda^{q-1}\exp\bigg\{-\bigg(1+\frac{r}{\rho(x_0)}\bigg)^{-(N_0+1)\theta_1}
\frac{C_2\delta\lambda}{\|f\|_{\mathrm{BMO}_{\rho,\theta_1}}}\bigg\}d\lambda\bigg)^{1/q}\nonumber\\
&\times\bigg(1+\frac{r}{\rho(x_0)}\bigg)^{\eta/q}.\nonumber\\
\end{align*}
By making the substitution $\nu=\big(1+\frac{r}{\rho(x_0)}\big)^{-(N_0+1)\theta_1}
\frac{C_2\delta\lambda}{\|f\|_{\mathrm{BMO}_{\rho,\theta_1}}}$, we can see that
\begin{align}\label{21}
&\bigg(\frac{1}{\omega^q(\mathcal{Q})}\int_{\mathcal{Q}}|f(x)-f_{\mathcal{Q}}|^q\omega(x)^q\,dx\bigg)^{1/q}\nonumber\\
&\leq C\cdot\frac{C_1^{\delta/q}}{C_2}\bigg(1+\frac{r}{\rho(x_0)}\bigg)^{(N_0+1)\theta_1+\eta/q}\|f\|_{\mathrm{BMO}_{\rho,\theta_1}}
\times\bigg(\int_0^{\infty}\nu^{q-1}e^{-\nu}\,d\nu\bigg)^{1/q}\nonumber\\
&\leq C\cdot\frac{C_1^{\delta/q}}{C_2}\bigg(1+\frac{r}{\rho(x_0)}\bigg)^{(N_0+1)\theta_1+\eta/q}\|f\|_{\mathrm{BMO}_{\rho,\theta_1}}.
\end{align}
This yields the desired estimate. Let us now turn to the proof of (2). When $1<p<\infty$, it then follows directly from the H\"{o}lder inequality that
\begin{align}\label{22}
\frac{1}{|\mathcal{Q}|}\int_{\mathcal{Q}}|f(x)-f_{\mathcal{Q}}|\,dx
&=\frac{1}{|\mathcal{Q}|}\int_{\mathcal{Q}}|f(x)-f_{\mathcal{Q}}|\omega(x)\cdot\omega(x)^{-1}\,dx\\
&\leq\frac{1}{|\mathcal{Q}|}\bigg(\int_{\mathcal{Q}}|f(x)-f_{\mathcal{Q}}|^p\omega(x)^p\,dx\bigg)^{1/p}
\bigg(\int_{\mathcal{Q}}\omega(x)^{-{p'}}\,dx\bigg)^{1/{p'}}\nonumber.
\end{align}
Moreover, by using the H\"{o}lder inequality again, we can see that when $1\leq p<q$,
\begin{equation}\label{above}
\bigg(\frac{1}{|\mathcal{Q}|}\int_{\mathcal{Q}}|f(x)-f_{\mathcal{Q}}|^p\omega(x)^p\,dx\bigg)^{1/p}
\leq\bigg(\frac{1}{|\mathcal{Q}|}\int_{\mathcal{Q}}|f(x)-f_{\mathcal{Q}}|^q\omega(x)^q\,dx\bigg)^{1/q}.
\end{equation}
Substituting the above inequality into \eqref{22}, we thus obtain
\begin{align}\label{221}
\frac{1}{|\mathcal{Q}|}\int_{\mathcal{Q}}|f(x)-f_{\mathcal{Q}}|\,dx
&\leq \frac{|\mathcal{Q}|^{1/p-1/q}}{|\mathcal{Q}|}\bigg(\int_{\mathcal{Q}}|f(x)-f_{\mathcal{Q}}|^q\omega(x)^q\,dx\bigg)^{1/q}
\bigg(\int_{\mathcal{Q}}\omega(x)^{-{p'}}\,dx\bigg)^{1/{p'}}\nonumber\\
&\leq C\bigg(1+\frac{r}{\rho(x_0)}\bigg)^{\theta_1-\theta_2}\nonumber\\
&\times\bigg(\frac{1}{|\mathcal{Q}|}\int_{\mathcal{Q}}\omega(x)^q\,dx\bigg)^{1/q}
\bigg(\frac{1}{|\mathcal{Q}|}\int_{\mathcal{Q}}\omega(x)^{-{p'}}\,dx\bigg)^{1/{p'}}\nonumber\\
&\leq C[\omega]_{A^{\rho,\theta_2}_{p,q}}\bigg(1+\frac{r}{\rho(x_0)}\bigg)^{\theta_1},
\end{align}
where in the last two inequalities we have used \eqref{assum2} and the definition of $A^{\rho,\theta_2}_{p,q}$, respectively. When $p=1$ and $1<q<\infty$, then we have
\begin{align*}
\frac{1}{|\mathcal{Q}|}\int_{\mathcal{Q}}|f(x)-f_{\mathcal{Q}}|\,dx
&=\frac{1}{|\mathcal{Q}|}\int_{\mathcal{Q}}|f(x)-f_{\mathcal{Q}}|\omega(x)\cdot\omega(x)^{-1}\,dx\nonumber\\
&\leq\frac{1}{|\mathcal{Q}|}\bigg(\int_{\mathcal{Q}}|f(x)-f_{\mathcal{Q}}|\omega(x)\,dx\bigg)
\bigg(\underset{x\in \mathcal{Q}}{\mbox{ess\,sup}}\;\omega(x)^{-1}\bigg).\nonumber\\
\end{align*}
From the previous estimate \eqref{above}(with $p=1$), it actually follows that
\begin{align}\label{222}
\frac{1}{|\mathcal{Q}|}\int_{\mathcal{Q}}|f(x)-f_{\mathcal{Q}}|\,dx
&\leq \frac{|\mathcal{Q}|^{1-1/q}}{|\mathcal{Q}|}\bigg(\int_{\mathcal{Q}}|f(x)-f_{\mathcal{Q}}|^q\omega(x)^q\,dx\bigg)^{1/q}
\bigg(\underset{x\in \mathcal{Q}}{\mbox{ess\,sup}}\;\omega(x)^{-1}\bigg)\nonumber\\
&\leq C\bigg(1+\frac{r}{\rho(x_0)}\bigg)^{\theta_1-\theta_2}\nonumber\\
&\times\bigg(\frac{1}{|\mathcal{Q}|}\int_{\mathcal{Q}}\omega(x)^q\,dx\bigg)^{1/q}
\bigg(\underset{x\in \mathcal{Q}}{\mbox{ess\,inf}}\;\omega(x)\bigg)^{-1}\nonumber\\
&\leq C[\omega]_{A^{\rho,\theta_2}_{1,q}}\bigg(1+\frac{r}{\rho(x_0)}\bigg)^{\theta_1},
\end{align}
where in the last two inequalities we have used \eqref{assum2} and the definition of $A^{\rho,\theta_2}_{1,q}$, respectively. Collecting the above estimates \eqref{221} and \eqref{222}, we finish the proof of Theorem \ref{main2}.
\end{proof}
For any cube $\mathcal{Q}$ (or ball $\mathcal{B}$) in $\mathbb R^d$, by using H\"{o}lder's inequality ($1<p<\infty$), we have
\begin{equation*}
|\mathcal{Q}|=\int_{\mathcal{Q}}\omega(x)\cdot\omega(x)^{-1}\,dx
\leq\bigg(\int_{\mathcal{Q}}\omega(x)^p\,dx\bigg)^{1/p}\bigg(\int_{\mathcal{Q}}\omega(x)^{-p'}\,dx\bigg)^{1/{p'}}.
\end{equation*}
By the definition of $A^{\rho,\theta_2}_{p,q}$ weights, we get
\begin{equation*}
\begin{split}
\bigg(\int_{\mathcal{Q}}\omega(x)^q\,dx\bigg)^{1/q}
\bigg(\int_{\mathcal{Q}}\omega(x)^{-{p'}}\,dx\bigg)^{1/{p'}}
&\leq [\omega]_{A^{\rho,\theta_2}_{p,q}}|\mathcal{Q}|^{1/q+1/{p'}}\bigg(1+\frac{r}{\rho(x_0)}\bigg)^{\theta_2}.
\end{split}
\end{equation*}
Consequently,
\begin{equation}\label{wangh1}
\begin{split}
\bigg(\int_{\mathcal{Q}}\omega(x)^q\,dx\bigg)^{1/q}
&\leq[\omega]_{A^{\rho,\theta_2}_{p,q}}\frac{|\mathcal{Q}|^{1/q+1/{p'}}}{|\mathcal{Q}|}\bigg(1+\frac{r}{\rho(x_0)}\bigg)^{\theta_2}
\bigg(\int_{\mathcal{Q}}\omega(x)^p\,dx\bigg)^{1/p}\\
&=[\omega]_{A^{\rho,\theta_2}_{p,q}}\bigg(1+\frac{r}{\rho(x_0)}\bigg)^{\theta_2}
|\mathcal{Q}|^{1/q-1/p}\bigg(\int_{\mathcal{Q}}\omega(x)^p\,dx\bigg)^{1/p}.
\end{split}
\end{equation}
We remark that the above estimate also holds for the case $p=1$ and $\omega\in A^{\rho,\theta_2}_{1,q}(\mathbb R^d)$. Indeed, it is immediate that by definition
\begin{equation*}
\begin{split}
\bigg(\int_{\mathcal{Q}}\omega(x)^q\,dx\bigg)^{1/q}
&\leq [\omega]_{A^{\rho,\theta_2}_{1,q}}\bigg(1+\frac{r}{\rho(x_0)}\bigg)^{\theta_2}
|\mathcal{Q}|^{1/q}\underset{x\in \mathcal{Q}}{\mbox{ess\,inf}}\;\omega(x)\\
&\leq [\omega]_{A^{\rho,\theta_2}_{1,q}}\bigg(1+\frac{r}{\rho(x_0)}\bigg)^{\theta_2}
|\mathcal{Q}|^{1/q-1}\bigg(\int_{\mathcal{Q}}\omega(x)\,dx\bigg).
\end{split}
\end{equation*}
On the other hand, it follows directly from H\"{o}lder's inequality ($1\leq p<q$) that
\begin{equation*}
\bigg(\frac{1}{|\mathcal{Q}|}\int_{\mathcal{Q}}\omega(x)^p\,dx\bigg)^{1/p}
\leq\bigg(\frac{1}{|\mathcal{Q}|}\int_{\mathcal{Q}}\omega(x)^q\,dx\bigg)^{1/q},
\end{equation*}
which implies that
\begin{equation}\label{wangh2}
\bigg(\int_{\mathcal{Q}}\omega(x)^q\,dx\bigg)^{1/q}\geq|\mathcal{Q}|^{1/q-1/p}\bigg(\int_{\mathcal{Q}}\omega(x)^p\,dx\bigg)^{1/p}.
\end{equation}
As a consequence of \eqref{wangh1} and \eqref{wangh2}, we then obtain the following conclusions.
\begin{corollary}\label{cor1}
Let $1\leq p<q<\infty$ and $\omega\in A^{\rho,\theta_2}_{p,q}(\mathbb R^d)$ with $0<\theta_2<\infty$. Then the following
statements are true.
\begin{enumerate}
\item If $f\in \mathrm{BMO}_{\rho,\theta_1}(\mathbb R^d)$ with $0<\theta_1<\infty$, then for any cube $\mathcal{Q}=Q(x_0,r)\subset\mathbb R^d$,
\begin{equation*}
\frac{|\mathcal{Q}|^{1/p-1/q}}{[\omega^p(\mathcal{Q})]^{1/p}}\bigg(\int_{\mathcal{Q}}|f(x)-f_{\mathcal{Q}}|^q\omega(x)^q\,dx\bigg)^{1/q}
\leq C\bigg(1+\frac{r}{\rho(x_0)}\bigg)^{(N_0+1)\theta_1+\theta_2+\eta/p}\|f\|_{\mathrm{BMO}_{\rho,\theta_1}}.
\end{equation*}
\item Conversely, if there exists a constant $C>0$ such that for any cube $\mathcal{Q}=Q(x_0,r)\subset\mathbb R^d$,
\begin{equation*}
\frac{|\mathcal{Q}|^{1/p-1/q}}{[\omega^p(\mathcal{Q})]^{1/p}}\bigg(\int_{\mathcal{Q}}|f(x)-f_{\mathcal{Q}}|^q\omega(x)^q\,dx\bigg)^{1/q}\leq C\bigg(1+\frac{r}{\rho(x_0)}\bigg)^{\theta_1-\theta_2}
\end{equation*}
holds for some $\theta_1>0$, then $f\in \mathrm{BMO}_{\rho,\theta_1}(\mathbb R^d)$, and
\begin{equation*}
\|f\|_{\mathrm{BMO}_{\rho,\theta_1}}\leq C[\omega]_{A^{\rho,\theta_2}_{p,q}}.
\end{equation*}
\end{enumerate}
\end{corollary}

\begin{theorem}\label{main3}
Let $1\leq p<\infty$ and $\omega\in A^{\rho,\theta_2}_p(\mathbb R^d)$ with $0<\theta_2<\infty$. Then the following
statements are true.
\begin{enumerate}
\item If $f\in \mathrm{Lip}_{\beta}^{\rho,\theta_1}(\mathbb R^d)$ with $0<\beta<1$ and $0<\theta_1<\infty$, then for any ball $\mathcal{B}=B(x_0,r)\subset\mathbb R^d$,
\begin{equation*}
\frac{1}{|\mathcal{B}|^{\beta/d}}\bigg(\frac{1}{\omega(\mathcal{B})}\int_{\mathcal{B}}|f(x)-f_{\mathcal{B}}|^p\omega(x)\,dx\bigg)^{1/p}
\leq C\bigg(1+\frac{r}{\rho(x_0)}\bigg)^{(N_0+1)\theta_1}\|f\|_{\mathrm{Lip}_{\beta}^{\rho,\theta_1}}.
\end{equation*}
\item Conversely, if there exists a constant $C>0$ such that for any ball $\mathcal{B}=B(x_0,r)\subset\mathbb R^d$ and $0<\beta<1$,
\begin{equation*}
\frac{1}{|\mathcal{B}|^{\beta/d}}\bigg(\frac{1}{\omega(\mathcal{B})}\int_{\mathcal{B}}|f(x)-f_{\mathcal{B}}|^p\omega(x)\,dx\bigg)^{1/p}\leq C\bigg(1+\frac{r}{\rho(x_0)}\bigg)^{\theta_1-\theta_2}
\end{equation*}
holds for some $\theta_1>0$, then $f\in \mathrm{Lip}_{\beta}^{\rho,\theta_1}(\mathbb R^d)$, and
\begin{equation*}
\|f\|_{\mathrm{Lip}_{\beta}^{\rho,\theta_1}}\leq C[\omega]_{A^{\rho,\theta_2}_p}.
\end{equation*}
\end{enumerate}
\end{theorem}

\begin{theorem}\label{main4}
Let $1\leq p<q<\infty$ and $\omega\in A^{\rho,\theta_2}_{p,q}(\mathbb R^d)$ with $0<\theta_2<\infty$. Then the following
statements are true.
\begin{enumerate}
\item If $f\in \mathrm{Lip}_{\beta}^{\rho,\theta_1}(\mathbb R^d)$ with $0<\beta<1$ and $0<\theta_1<\infty$, then for any ball $\mathcal{B}=B(x_0,r)\subset\mathbb R^d$,
\begin{equation*}
\frac{1}{|\mathcal{B}|^{\beta/d}}\bigg(\frac{1}{\omega^q(\mathcal{B})}\int_{\mathcal{B}}|f(x)-f_{\mathcal{B}}|^q\omega(x)^q\,dx\bigg)^{1/q}
\leq C\bigg(1+\frac{r}{\rho(x_0)}\bigg)^{(N_0+1)\theta_1}\|f\|_{\mathrm{Lip}_{\beta}^{\rho,\theta_1}}.
\end{equation*}
\item Conversely, if there exists a constant $C>0$ such that for any ball $\mathcal{B}=B(x_0,r)\subset\mathbb R^d$ and $0<\beta<1$,
\begin{equation*}
\frac{1}{|\mathcal{B}|^{\beta/d}}\bigg(\frac{1}{\omega^q(\mathcal{B})}\int_{\mathcal{B}}|f(x)-f_{\mathcal{B}}|^q\omega(x)^q\,dx\bigg)^{1/q}\leq C\bigg(1+\frac{r}{\rho(x_0)}\bigg)^{\theta_1-\theta_2}
\end{equation*}
holds for some $\theta_1>0$, then $f\in \mathrm{Lip}_{\beta}^{\rho,\theta_1}(\mathbb R^d)$, and
\begin{equation*}
\|f\|_{\mathrm{Lip}_{\beta}^{\rho,\theta_1}}\leq C[\omega]_{A^{\rho,\theta_2}_{p,q}}.
\end{equation*}
\end{enumerate}
\end{theorem}

\begin{proof}[Proof of Theorem \ref{main3}]
Following the same arguments as in the proof of Theorem \ref{main1}, we can also prove the second part (2). We only need to show the first part (1).
Fix $x_0\in\mathbb R^d$, let $\mathcal{B}=B(x_0,r)$ be the ball centered at $x_0$ and radius $r$. By using Minkowski's inequality, we get
\begin{equation*}
\begin{split}
&\frac{1}{|\mathcal{B}|^{\beta/d}}\bigg(\frac{1}{\omega(\mathcal{B})}\int_{\mathcal{B}}|f(x)-f_{\mathcal{B}}|^p\omega(x)\,dx\bigg)^{1/p}\\
&\leq\frac{1}{|\mathcal{B}|^{\beta/d}}\bigg(\frac{1}{\omega(\mathcal{B})}\int_{\mathcal{B}}|f(x)-f(x_0)|^p\omega(x)\,dx\bigg)^{1/p}
+\frac{1}{|\mathcal{B}|^{\beta/d}}\bigg(\frac{1}{\omega(\mathcal{B})}\int_{\mathcal{B}}|f(x_0)-f_{\mathcal{B}}|^p\omega(x)\,dx\bigg)^{1/p}\\
&\leq\frac{1}{|\mathcal{B}|^{\beta/d}}\bigg(\frac{1}{\omega(\mathcal{B})}\int_{\mathcal{B}}|f(x)-f(x_0)|^p\omega(x)\,dx\bigg)^{1/p}
+\frac{1}{|\mathcal{B}|^{\beta/d}}\bigg(\frac{1}{|\mathcal{B}|}\int_{\mathcal{B}}|f(x)-f(x_0)|\,dx\bigg).
\end{split}
\end{equation*}
In view of Lemma \ref{beta}, one can see that for any $x\in \mathcal{B}$,
\begin{equation*}
\begin{split}
|f(x)-f(x_0)|&\leq C\|f\|_{\mathrm{Lip}_{\beta}^{\rho,\theta_1}}|x-x_0|^{\beta}
\bigg(1+\frac{|x-x_0|}{\rho(x)}+\frac{|x-x_0|}{\rho(x_0)}\bigg)^{\theta_1}\\
&\leq C|\mathcal{B}|^{\beta/d}\|f\|_{\mathrm{Lip}_{\beta}^{\rho,\theta_1}}
\bigg(1+\frac{r}{\rho(x)}+\frac{r}{\rho(x_0)}\bigg)^{\theta_1}.
\end{split}
\end{equation*}
This, together with the estimate \eqref{wangh3}, gives us that
\begin{equation*}
|f(x)-f(x_0)|\leq C|\mathcal{B}|^{\beta/d}\|f\|_{\mathrm{Lip}_{\beta}^{\rho,\theta_1}}
\bigg(1+\frac{r}{\rho(x_0)}\bigg)^{(N_0+1)\theta_1}.
\end{equation*}
Therefore,
\begin{equation*}
\frac{1}{|\mathcal{B}|^{\beta/d}}\bigg(\frac{1}{\omega(\mathcal{B})}\int_{\mathcal{B}}|f(x)-f_{\mathcal{B}}|^p\omega(x)\,dx\bigg)^{1/p}
\leq C\bigg(1+\frac{r}{\rho(x_0)}\bigg)^{(N_0+1)\theta_1}\|f\|_{\mathrm{Lip}_{\beta}^{\rho,\theta_1}}.
\end{equation*}
This completes the proof of Theorem \ref{main3}.
\end{proof}

\begin{proof}[Proof of Theorem \ref{main4}]
Following along the same lines as that of Theorem \ref{main2}, we can also prove the second part (2). So we only need to show the first part (1). For an arbitrary fixed ball $\mathcal{B}=B(x_0,r)$, by the Minkowski inequality, we have
\begin{equation*}
\begin{split}
&\frac{1}{|\mathcal{B}|^{\beta/d}}\bigg(\frac{1}{\omega^q(\mathcal{B})}\int_{\mathcal{B}}|f(x)-f_{\mathcal{B}}|^q\omega(x)^q\,dx\bigg)^{1/q}\\
&\leq\frac{1}{|\mathcal{B}|^{\beta/d}}\bigg(\frac{1}{\omega^q(\mathcal{B})}\int_{\mathcal{B}}|f(x)-f(x_0)|^q\omega(x)^q\,dx\bigg)^{1/q}
+\frac{1}{|\mathcal{B}|^{\beta/d}}\bigg(\frac{1}{\omega^q(\mathcal{B})}\int_{\mathcal{B}}|f(x_0)-f_{\mathcal{B}}|^q\omega(x)^q\,dx\bigg)^{1/q}\\
&\leq\frac{1}{|\mathcal{B}|^{\beta/d}}\bigg(\frac{1}{\omega^q(\mathcal{B})}\int_{\mathcal{B}}|f(x)-f(x_0)|^q\omega(x)^q\,dx\bigg)^{1/q}
+\frac{1}{|\mathcal{B}|^{\beta/d}}\bigg(\frac{1}{|\mathcal{B}|}\int_{\mathcal{B}}|f(x)-f(x_0)|\,dx\bigg).
\end{split}
\end{equation*}
Arguing as in the proof of Theorem \ref{main3}, we can also obtain analogous estimate below.
\begin{equation*}
\frac{1}{|\mathcal{B}|^{\beta/d}}\bigg(\frac{1}{\omega^q(\mathcal{B})}\int_{\mathcal{B}}|f(x)-f_{\mathcal{B}}|^q\omega(x)^q\,dx\bigg)^{1/q}
\leq C\bigg(1+\frac{r}{\rho(x_0)}\bigg)^{(N_0+1)\theta_1}\|f\|_{\mathrm{Lip}_{\beta}^{\rho,\theta_1}}.
\end{equation*}
This concludes the proof of Theorem \ref{main4}.
\end{proof}
In view of the estimates \eqref{wangh1} and \eqref{wangh2}, we immediately obtain the following results.
\begin{corollary}\label{cor2}
Let $1\leq p<q<\infty$ and $\omega\in A^{\rho,\theta_2}_{p,q}(\mathbb R^d)$ with $0<\theta_2<\infty$. Then the following
statements are true.
\begin{enumerate}
\item If $f\in \mathrm{Lip}_{\beta}^{\rho,\theta_1}(\mathbb R^d)$ with $0<\beta<1$ and $0<\theta_1<\infty$, then for any ball $\mathcal{B}=B(x_0,r)\subset\mathbb R^d$,
\begin{equation*}
\frac{|\mathcal{B}|^{1/p-1/q-\beta/d}}{[\omega^p(\mathcal{B})]^{1/p}}
\bigg(\int_{\mathcal{B}}|f(x)-f_{\mathcal{B}}|^q\omega(x)^q\,dx\bigg)^{1/q}
\leq C\bigg(1+\frac{r}{\rho(x_0)}\bigg)^{(N_0+1)\theta_1+\theta_2}\|f\|_{\mathrm{Lip}_{\beta}^{\rho,\theta_1}}.
\end{equation*}
\item Conversely, if there exists a constant $C>0$ such that for any ball $\mathcal{B}=B(x_0,r)\subset\mathbb R^d$ and $0<\beta<1$,
\begin{equation*}
\frac{|\mathcal{B}|^{1/p-1/q-\beta/d}}{[\omega^p(\mathcal{B})]^{1/p}}
\bigg(\int_{\mathcal{B}}|f(x)-f_{\mathcal{B}}|^q\omega(x)^q\,dx\bigg)^{1/q}\leq C\bigg(1+\frac{r}{\rho(x_0)}\bigg)^{\theta_1-\theta_2}
\end{equation*}
holds for some $\theta_1>0$, then $f\in \mathrm{Lip}_{\beta}^{\rho,\theta_1}(\mathbb R^d)$, and
\begin{equation*}
\|f\|_{\mathrm{Lip}_{\beta}^{\rho,\theta_1}}\leq C[\omega]_{A^{\rho,\theta_2}_{p,q}}.
\end{equation*}
\end{enumerate}
\end{corollary}

\textbf{Concluding remarks}. Summarizing the estimates derived above, by definition, we then have the following conclusions.
\begin{corollary}\label{cor3}
Let $1\leq p<\infty$ and $\omega\in A^{\rho,\infty}_p(\mathbb R^d)$. Then the following
statements are true.
\begin{enumerate}
\item $f\in \mathrm{BMO}_{\rho,\infty}(\mathbb R^d)$ if and only if there exists a constant $C>0$ such that, for any cube $\mathcal{Q}=Q(x_0,r)\subset\mathbb R^d$,
\begin{equation*}
\bigg(\frac{1}{\omega(\mathcal{Q})}\int_{\mathcal{Q}}|f(x)-f_{\mathcal{Q}}|^p\omega(x)\,dx\bigg)^{1/p}
\leq C\bigg(1+\frac{r}{\rho(x_0)}\bigg)^{\mathcal{N}}
\end{equation*}
holds true for some $\mathcal{N}>0$. 
\item $f\in \mathrm{Lip}_{\beta}^{\rho,\infty}(\mathbb R^d)$ with $0<\beta<1$ if and only if there exists a constant $C>0$ such that, for any ball $\mathcal{B}=B(x_0,r)\subset\mathbb R^d$,
\begin{equation*}
\frac{1}{|\mathcal{B}|^{\beta/d}}\bigg(\frac{1}{\omega(\mathcal{B})}\int_{\mathcal{B}}|f(x)-f_{\mathcal{B}}|^p\omega(x)\,dx\bigg)^{1/p}
\leq C\bigg(1+\frac{r}{\rho(x_0)}\bigg)^{\mathcal{N'}}
\end{equation*}
holds true for some $\mathcal{N'}>0$. 
\end{enumerate}
\end{corollary}

\begin{corollary}\label{cor4}
Let $1\leq p<q<\infty$ and $\omega\in A^{\rho,\infty}_{p,q}(\mathbb R^d)$. Then the following statements are true.
\begin{enumerate}
  \item $f\in \mathrm{BMO}_{\rho,\infty}(\mathbb R^d)$ if and only if there exists a constant $C>0$ such that, for any cube $\mathcal{Q}=Q(x_0,r)\subset\mathbb R^d$,
\begin{equation*}
\bigg(\frac{1}{\omega^q(\mathcal{Q})}\int_{\mathcal{Q}}|f(x)-f_{\mathcal{Q}}|^q\omega(x)^q\,dx\bigg)^{1/q}\leq C\bigg(1+\frac{r}{\rho(x_0)}\bigg)^{\mathcal{N}}
\end{equation*}
or
\begin{equation*}
\frac{|\mathcal{Q}|^{1/p-1/q}}{[\omega^p(\mathcal{Q})]^{1/p}}\bigg(\int_{\mathcal{Q}}|f(x)-f_{\mathcal{Q}}|^q\omega(x)^q\,dx\bigg)^{1/q}\leq C\bigg(1+\frac{r}{\rho(x_0)}\bigg)^{\mathcal{N}} 
\end{equation*}
holds true for some $\mathcal{N}>0$.
  \item $f\in \mathrm{Lip}_{\beta}^{\rho,\infty}(\mathbb R^d)$ with $0<\beta<1$ if and only if there exists a constant $C>0$ such that, for any ball $\mathcal{B}=B(x_0,r)\subset\mathbb R^d$,
\begin{equation*}
\frac{1}{|\mathcal{B}|^{\beta/d}}\bigg(\frac{1}{\omega^q(\mathcal{B})}\int_{\mathcal{B}}|f(x)-f_{\mathcal{B}}|^q\omega(x)^q\,dx\bigg)^{1/q}\leq C\bigg(1+\frac{r}{\rho(x_0)}\bigg)^{\mathcal{N'}}
\end{equation*}
or
\begin{equation*}
\frac{|\mathcal{B}|^{1/p-1/q-\beta/d}}{[\omega^p(\mathcal{B})]^{1/p}}
\bigg(\int_{\mathcal{B}}|f(x)-f_{\mathcal{B}}|^q\omega(x)^q\,dx\bigg)^{1/q}\leq C\bigg(1+\frac{r}{\rho(x_0)}\bigg)^{\mathcal{N'}}
\end{equation*}
holds true for some $\mathcal{N'}>0$. 
\end{enumerate}
\end{corollary}

{\bf Acknowledgments.}
This work was supported by Natural Science Foundation of China under Grant XJEDU2020Y002 and 2022D01C407.

\bibliographystyle{amsplain}

\end{document}